\documentclass[a4paper,12pt]{amsart}
\usepackage[top=3cm,bottom=3cm,outer=3cm,inner=3cm,marginpar=3cm]{geometry}

\title[Diederich--Forn{\ae}ss and Steinness indices]{Semi-continuity of the Diederich-Forn{\ae}ss and Steinness indices}

\author{Young-Jun Choi}
\address{Department of Mathematics, Pusan National University, 2, Busandaehak-ro 63beon-gil, Geumjeong-gu, Busan 46241, Republic of Korea}
\email{youngjun.choi@pusan.ac.kr}
\thanks{}

\author{Jihun Yum}
\address{Department of Mathematics, Pusan National University, 2, Busandaehak-ro 63beon-gil, Geumjeong-gu, Busan 46241, Republic of Korea}
\email{jihun0224@pusan.ac.kr}
\thanks{}

\keywords{Diederich-Forn{\ae}ss index, Steinness index, deformation of pseudoconex domains}
\subjclass[2010]{32T27, 32U10, 32G05}
\date{\today}

\usepackage{amsmath,amsthm,amssymb,latexsym}
\usepackage{mathrsfs}
\usepackage[abbrev]{amsrefs}

\usepackage{hyperref}
\usepackage[autostyle]{csquotes}
\usepackage{xcolor}

\newtheorem{thm}{Theorem}[section]
\newtheorem{lem}[thm]{Lemma}

\theoremstyle{definition}
\newtheorem{defn}[thm]{Definition}

\newtheorem{rmk}[thm]{Remark}

\numberwithin{equation}{section}

\newcommand\D{\mathbb{D}}

\newcommand{\ol}{\overline }

\def\RR{\mathbb{R}} 
\def\CC{\mathbb{C}} 
 
\def\O{\Omega} 
\def\Levi{\mathscr{L}} 
\def\Lie{\mathcal{L}}  
\def\Null{\mathcal{N}}    

\begin{document}

\begin{abstract}
In this paper, we prove the semi-continuity theorem of Diederich-Forn{\ae}ss index and Steinness index under a smooth deformation of pseudoconvex domains in Stein manifolds.
\end{abstract}

\maketitle

\section{\bf Introduction}

Let $\O$ be a relatively compact, Levi pseudoconvex domain in a complex manifold $X$ with $C^{\infty}$-smooth boundary. Let $\rho : X \rightarrow \RR$ be a smooth defining function of $\O$, i.e., $\O = \{z \in X : \rho(z) < 0 \}$ and $d\rho \neq 0$ on $\partial \O$. The Diederich-Forn{\ae}ss index $DF(\O)$ and the Steinness index $S(\O)$ of $\O$ are defined by  
\begin{align*}
DF(\O) &:= \sup_{\rho} \left\{ 0< \gamma < 1 : -(-\rho)^{\gamma} \text{ is strictly plurisubharmonic on } \O \cap W \right\}, \\
S(\O) &:=  \inf_{\rho} \left\{ \gamma > 1 : \rho^{\gamma} \text{ is strictly plurisubharmonic on } \overline{\O}^{\complement} \cap W \right\},	
\end{align*}
where the supremum and infimum are taken over all smooth defining function $\rho$, and $W$ is some neighborhood of $\partial\O$ that may depend on $\rho$ and $\gamma$. If such $\rho$ and $\gamma$ do not exist, we define $DF(\O)=0$ and $S(\O)=\infty$.
When the ambient space $X$ is a Stein manifold, $DF(\O) > 0$ implies the existence of a bounded strictly plurisubharmonic function on $\O$, i.e., $\O$ is hyperconvex, and $S(\O) < \infty$ implies that $\O$ admits a Stein neighborhood basis. 
In 1977, Diederich and Forn{\ae}ss (\cite{diederich-fornaess}) proved the positivity of $DF(\O)$ if the boundary is of $C^2$-smoothness.
The second named author (\cite{Yum1}) provided a necessary and sufficient condition for $S(\O) < \infty$.
\medskip

In this paper, we shall study the semi-continuity of the both indices of a relatively compact pseudoconvex domain with smooth boundary under a smooth deformation.
More precisely, a smooth deformation of a relatively compact domain $\Omega_0$ in a complex manifold $X_0$ is given as follows:

A surjective holomorphic submersion $\pi : (X,\Omega) \rightarrow\D$ from a complex manifold $X$ with a relatively compact domain $\Omega$ in $X$ to the unit disc $\D$ is called a \emph{smooth deformation of $\Omega_0$ in $X_0$ over $\D$} if

\begin{itemize}
\item $X_0 = \pi^{-1}(0)$ and $\Omega_0 = \Omega \cap X_0$.
\item $\Omega$ admits a defining function $\delta$ such that $\delta\vert_{X_t}$ is a defining function of $\Omega_t:=\Omega \cap \pi^{-1}(t)$ in $X_t:=\pi^{-1}(t)$ for $t\in\D$.
\end{itemize}
Note that the above conditions guarantee that every fiber $\Omega_t$ is diffeomorphic to $\Omega_0$. (cf, see \cite{Saeki}.)
\medskip

The main theorem of this paper is as follows:

\begin{thm} \label{thm:main}
	If every fiber $X_t$ is a Stein manifold, then 
	\[ 
		\liminf_{t \rightarrow 0} DF(\O_t) \ge DF(\O_0) \quad \text{ and } \quad 
		\limsup_{t \rightarrow 0} S(\O_t) \le S(\O_0).
	\]
In other words, $DF(\O_t)$ is lower semi-continuous and $S(\O_t)$ is upper semi-continuous at $t=0$, respectively.
\end{thm}

We would like to emphasize that the Steinness of each fiber $X_t$ is a superfluous condition.
In fact, the condition in Theorem \ref{thm:DF,S formulas}, which is our main tool, is sufficient. 
That is,  if each $\O_t$ admits a defining function $\rho_t$ such that either 
$\ol{\partial} \omega_{\rho_t} > 0$ on $\Null$ or $\ol{\partial} \omega_{\rho_t} < 0$ on $\Null$, then we have the same conclusion.
Here, $\omega_{\rho_t}$ is a D'Angelo $(1,0)$-form and $\Null$ is the kernel of the Levi form (see Section \ref{sec:Preliminaries}).
We refer readers to Corollary 5.4 and Remark 5.6 in \cite{Adachi-Yum} for the details.

\bigskip

\subsection*{\bf Acknowledgment}
The authors would like to thank A. Seo for suggesting the problem and M. Adachi for a useful comment on Lemma \ref{lem:null-space}.
This work was supported by the National Research Foundation
(NRF) of Korea grant funded by the Korea government (No. 2018R1C1B3005963).


\section{\bf Preliminaries} \label{sec:Preliminaries}

In this section, we introduce the D'Angelo $1$-form and the characterizations of Diederich-Forn{\ae}ss and Steinness indices by the D'Angelo $1$-form due to Adachi and the second named author (\cite{Adachi-Yum}).

First, we recall the definition of D'Angelo 1-form, which was introduced by D'Angelo (\cite{D'Angelo0}, \cite{D'Angelo}), and developed by Boas and Straube (\cite{Boas-Straube}). 
Let $\O$ be a relatively compact, Levi pseudoconvex domain in a complex manifold $X$ with $C^{\infty}$-smooth boundary. Let $\rho$ be a smooth defining function of $\O$.
Denote the kernel of the Levi form by $\mathcal{N} = \bigcup_{p \in M} \mathcal{N}_p \subset T^{1,0}(\partial \O)$ where
\[
\mathcal{N}_p := \{ L_p \in T^{1,0}_p(\partial \O) \mid \Levi_{\rho}(L_p, L'_p) = 0 \quad \forall L_p' \in T^{1,0}_p(\partial \O) \}
\]
for each $p \in \partial \O$, where $\Levi_{\rho}$ is the Levi-form of $\rho$. 
Note that the kernel $\Null$ does not depend on a defining function $\rho$, and $\mathcal{N}_p = \{ L_p \in T^{1,0}_p(\partial \O) \mid \Levi_{\rho}(L_p, L_p) = 0 \}$
when $\O$ is a Levi pseudoconvex domain.
Define 
\[ 
\eta_{\rho} := \frac{1}{2}\left( \partial \rho - \overline{\partial} \rho \right),
\]
which is a purely imaginary, non-vanishing 1-form on $\partial \O$ that annihilates $T^{1,0}(\partial \O) \oplus T^{0,1}(\partial \O)$.
Let $T_{\rho} \in \Gamma(\CC \otimes T(\partial\O))$ be a purely imaginary, non-vanishing, smooth vector field on $\partial \O$ such that $\eta_{\rho}(T_{\rho}) = 1$. 
Then $T_{\rho}$ yields a decomposition
\[
\CC \otimes T(\partial\O) = T^{1,0}(\partial\O) \oplus T^{0,1}(\partial\O) \oplus \CC T_{\rho}.
\]
We call such $T_{\rho}$ a \emph{transversal vector field} normalized with respect to $\eta_{\rho}$.
Denote $\eta_{\rho}$ and $T_{\rho}$ by $\eta$ and $T$, respectively, if there is no ambiguity.

\begin{defn} \label{def D'Angelo 1-form}
	A {\it D'Angelo 1-form} $\alpha_{\rho}$ of $\rho$ on $\partial \O$ is defined by 
	$$ \alpha_{\rho} := - \Lie_{T_{\rho}} \eta_{\rho} ,$$
	where $\Lie_{T_{\rho}}$ is the Lie derivative in the direction of $T_\rho$. 
	A {\it D'Angelo $(1,0)$-form} $\omega_{\rho} := \pi_{1,0} \alpha_{\rho}$ is the projection of $\alpha_{\rho}$ onto its $(1,0)$-component.
\end{defn}

\begin{rmk}
	In \cite{Adachi-Yum}, they defined the D'Angelo 1-form on a compact abstract CR manifold $M$ of hypersurface type without using a defining function. However, this definition is equivalent to Definition \ref{def D'Angelo 1-form} when $M$ bounds a relatively compact domain in a complex manifold. In this paper, we use Definition \ref{def D'Angelo 1-form} because we deal with only domains.
\end{rmk}

Note that although, for a defining function $\rho$, a transversal vector field $T_{\rho}$ normalized with respect to $\eta_{\rho}$ is not unique, $\omega_{\rho}$ and $\ol{\partial} \omega_{\rho}$ are well-defined on $\Null$, that is, they are independent of the choice of $T_{\rho}$ (see Lemma 2.5 and 2.6 in \cite{Adachi-Yum}).
We regard $\ol{\partial} \omega_{\rho}$ and $\omega_{\rho} \wedge \ol{\omega}_{\rho}$ as quadratic forms on $\Null$, i.e., $\ol{\partial} \omega_{\rho} > 0$ on $\Null$ means $\ol{\partial} \omega_{\rho}(L, \ol{L}) > 0$ for all $L \in \Null$.
The main tool we will use in Section \ref{sec:Semi-continuity of two indices} is the following.

\begin{thm}[Adachi, Yum \cite{Adachi-Yum}] \label{thm:DF,S formulas}
	Suppose that there exists a defining function $\rho_1$ of $\O$ such that $\ol{\partial} \omega_{\rho_1} > 0$ on $\Null$ or $\ol{\partial} \omega_{\rho_1} < 0$ on $\Null$. Then 
	\begin{align*}
	DF(\O) &= \sup_{\rho} \left\{ 0 < \gamma < 1 : \overline{\partial} \omega_{\rho}  - \frac{\gamma}{1 - \gamma}(\omega_{\rho} \wedge \overline{\omega}_{\rho}) > 0 \quad \text{ on } \Null  \right\}, \\
	S(\O) &= \inf_{\rho} \left\{ \gamma > 1 : -\overline{\partial} \omega_{\rho} - \frac{\gamma}{\gamma - 1}(\omega_{\rho} \wedge \overline{\omega}_{\rho}) > 0 \quad \text{ on } \Null  \right\}.
	\end{align*}
	If the supremum or infimum does not exist, then $DF(\O)=0$ or $S(\O)=\infty$, respectively. 
\end{thm}


\section{\bf Semi-continuity of two indices} \label{sec:Semi-continuity of two indices}

In this section, we prove the main theorem. 
Let $\pi : (X,\Omega) \rightarrow \D$ be a smooth deformation of $\O_0$ in $X_0$ over $\D$.
Let $\dim_{\CC}X = n+1$ and $\dim_{\CC}X_t = n$.
Let $\delta$ be a defining function of $\O \subset X$ such that $\delta_t := \delta|_{X_t}$ is also a defining function of $\O_t \subset X_t$ for each $t \in \D$, and $M := \partial \O$.
For a point $p \in \partial\O_t \subset M$, denote by 
\[
\Null_p := \{ L_p \in T^{1,0}_p(\partial \O_t) \mid \Levi_{\delta_t}(L_p, L'_p) = 0 \quad \forall L_p' \in T^{1,0}_p(\partial \O_t) \}.
\]

\begin{lem} \label{lem:null-space} 
	For $p \in \partial \O_0 \subset M$, suppose that $dim_{\CC} (\Null_p) = m$ $(1 \le m \le n-1)$. Then there exist a neighborhood $U_p$ of $p$ in $M$ and $m$ linearly independent smooth $(1,0)$ vector fields $L_1, \cdots, L_m$ on $U_p$ satisfying the following conditions:
	\begin{itemize}
		\item $\Null_p = \left< L_1(p), \cdots, L_m(p) \right> $, 
		\item For each $q \in U_p$, $\Null_q \subset \left< L_1(q), \cdots, L_m(q) \right> $.
	\end{itemize}
	
\end{lem}
\begin{proof}
	Choose a basis $\{ L_{1,p}, \cdots, L_{n-1,p} \}$  for $T^{1,0}_p(\partial \O_0)$ such that $\{ L_{1,p}, \cdots, L_{m,p} \}$ forms a basis for $\Null_p$. 
	Then, for each $j=1, \cdots, n-1$, we extend $L_{j,p}$ smoothly to a $(1,0)$ vector field $X_j$ on a neighborhood $U_p$ of $p$ in $M$ such that 
	$X_j(p) = L_{j,p}$ and $\{ X_1(q), \cdots, X_{n-1}(q) \}$ forms a basis for $T^{1,0}_q(\partial \O_t)$ for $q \in U_p$ with $\pi(q)=t$.
	This is possible because $\bigcup_{q \in \partial \O} T^{1,0}_q(\partial \O_t)$ is a vector subbundle of $T^{1,0}(\partial \O)$.
	Let $M(q)$ be the matrix representation of the Levi-form with respect to a basis $\{ X_j(q) \}^{n-1}_{j=1}$ at $q \in U_p$, i.e., 
	$M(q) = [ M_{i\overline{j}}(q) ]_{(n-1)\times(n-1)} := \left[ \Levi_{\delta_t}(X_i(q), X_j(q)) \right] $. Then $M(q)$ is a hermitian matrix and let  \\
	\[
	M(q) = 
	\begin{bmatrix}
		A(q) & B(q) \\
		B(q)^* & C(q) 
	\end{bmatrix}
	\]
	with blocks of size $m$ and $(n-m-1)$, where $B(q)^*$ is the conjugate transpose of $B(q)$.
	Also, by the construction, $A(p) = B(p) = 0$ and $C(p)$ is invertible.
	By the continuity, $C(q)$ is still invertible on $U_p$ (by shrinking $U_p$ if necessary).
	Let 
	$$\Psi(q) :=
	\begin{bmatrix}
		I_{m \times m} & 0 \\
		-C(q)^{-1} B(q)^* & C(q)^{-1} 
	\end{bmatrix}
	,$$
	where $I_{m \times m}$ is the identity matrix.
	Define vector fields $L_j(q)$ $(1 \le j \le n-1)$ on $U_p$ by
	\[
		\left[ L_1(q) \cdots L_{n-1}(q) \right] := \left[ X_1(q) \cdots X_{n-1}(q) \right] \times \overline{\Psi}(q),
	\]
	where $\times$ is the matrix multiplication.
	Then $\Null_p = \left< L_1(p), \cdots, L_m(p) \right> $ because $L_j(p) = X_j(p) = L_{j,p}$ for $j = 1, \cdots, m$, and
	the matrix representation of the Levi-form with respect to a basis $\{ L_j(q) \}^{n-1}_{j=1}$ is
	\begin{equation*} 
		\Psi(q)^* \times M(q) \times \Psi(q) = 
		\begin{bmatrix}
			A(q) - B(q)C(q)^{-1}B(q)^* & 0 \\
			0 &  C(q)^{-1}
		\end{bmatrix}
		.
	\end{equation*}
	Therefore, since $C(q)^{-1}$ is invertible, we conclude that $\Null_q \subset \left< L_1(q), \cdots, L_m(q) \right> $.
\end{proof}

\begin{proof}[Proof of Theorem \ref{thm:main}]
	Since each fiber $X_t$ is a Stein manifold, one may apply Theorem \ref{thm:DF,S formulas} for each $\O_t$ (see Corollary 5.4 in \cite{Adachi-Yum}).
	By Theorem \ref{thm:DF,S formulas}, for each $0 < \gamma < DF(\O_0)$, there exists a defining function $\rho : X_0 \rightarrow \RR$ of $\O_0 \subset X_0$ such that 
	\begin{equation} \label{inq:3.1}
		\overline{\partial} \omega_{\rho}  - \frac{\gamma}{1 - \gamma}(\omega_{\rho} \wedge \overline{\omega}_{\rho}) > 0 \quad \text{ on } \Null,
	\end{equation}
	where $\Null := \bigcup_{p \in \partial \O_0} \mathcal{N}_p \subset T^{1,0}(\partial \O_0)$.
	We will extend $\rho$ smoothly to $\widetilde{\rho} : \pi^{-1}(U_0)  \rightarrow \RR$ such that $\widetilde{\rho}_t := \widetilde{\rho}|_{X_t}$ is a defining function of $\O_t$ for each $t \in U_0$ and $\widetilde{\rho}_0 = \rho$, where $U_0$ is a neighborhood of $0 \in \D$. Since $\delta_0$ is also a defining function of $\O_0 \subset X_0$, there exists a smooth function $\psi \in C^{\infty}(X_0)$ such that $\rho = e^{\psi} \delta_0$. 
	Then since $X_0$ is a closed submanifold of $X$, we may extend $\psi$ to $\widetilde{\psi} \in C^{\infty}(\pi^{-1}(U_0))$ (by shrinking $U_0$ if necessary), and $\widetilde{\rho} := e^{\widetilde{\psi}} \delta$ is the desired function. We denote $\widetilde{\rho}$ by $\rho$ again, and $\omega_{\rho_t}$ by $\omega_{t}$.
	
	Now for $p \in \partial \O_0$, suppose that $dim_{\CC} (\Null_p) = m >0 $. Then by Lemma \ref{lem:null-space},
	there exist a neighborhood $U_p$ of $p$ in $M$ and $m$ linearly independent smooth $(1,0)$ vector fields $L_1, \cdots, L_m$ on $U_p$ satisfying the following conditions:
	\begin{itemize}
		\item $\Null_p = \left< L_1(p), \cdots, L_m(p) \right> $, 
		\item For each $q \in U_p$, $\Null_q \subset \left< L_1(q), \cdots, L_m(q) \right> $.
	\end{itemize}
	Then, by the continuity, the inequality (\ref{inq:3.1}) implies that there exists a neighborhood $V_p \subset U_p$ such that
	\[
		\overline{\partial} \omega_{t}  - \frac{\gamma}{1 - \gamma}(\omega_{t} \wedge \overline{\omega}_{t}) > 0 \quad \text{ on } \left< L_1(q), \cdots, L_m(q) \right>
	\]
	for all $q \in V_p$ with $\pi(q)=t$.
	Since $\Null_q \subset \left< L_1(q), \cdots, L_m(q) \right> $,
	\[
		\overline{\partial} \omega_{t}  - \frac{\gamma}{1 - \gamma}(\omega_{t} \wedge \overline{\omega}_{t}) > 0 \quad \text{ on } \Null_q
	\]
	for all $q \in V_p$ with $\pi(q)=t$. Therefore, since the set $\{ p \in \partial \O_0 : \dim_{\CC}(\Null_p) > 0 \}$ is compact, we conclude that
	$$\liminf_{t \rightarrow 0} DF(\O_t) \ge DF(\O_0),$$ 
	by applying Theorem \ref{thm:DF,S formulas} again. 
	$\limsup_{t \rightarrow 0} S(\O_t) \le S(\O_0)$ follows from the same argument as above. 
\end{proof}


\section{\bf Example}

In 1977, Diederich and Forn{\ae}ss (\cite{diederich-fornaess2}) constructed a 1-parameter family of bounded, pseudoconvex domains $\O_{\beta}$ $(\beta > \frac{\pi}{2})$, called worm domains, in $\CC^2$ with $C^{\infty}$-smooth boundaries. They showed that the Diederich-Forn{\ae}ss indices of worm domains are non-trivial, i.e., $0 < DF(\O_{\beta}) < 1$, for all $\beta$, and $\O_{\beta}$ does not admit a Stein neighborhood basis for some $\beta$. 
Recently, Liu (\cite{Liu1}) calculated the exact value of the $DF(\O_{\beta})$. Also, the second named author (\cite{Yum1}) calculated the exact value of the $S(\O_{\beta})$ and found the following relation between two indices for worm domains:
\[
	\frac{1}{DF(\O_{\beta})} + \frac{1}{S(\O_{\beta})} = 2,
\]
whenever $DF(\O_{\beta}) > 0$ and $S(\O_{\beta}) < \infty$.

In this section, we give an explicit example, by modifying worm domains, which shows the Diederich-Forn{\ae}ss and Steinness indices do not admit upper semi-continuity and lower semi-continuity in general, respectively. 
We first recall the $\beta$-worm domain.

\begin{defn} \label{worm defn}
	The {\it $\beta$-worm domain} $D_{\beta}$ $(\beta > \frac{\pi}{2})$ is defined by
	$$ D_{\beta} := \left\lbrace (z,w)\in \CC^2 : \left| z - e^{i \log|w|^2} \right|^2 - (1 - \phi_{\beta}(\log|w|^2) ) < 0  \right\rbrace $$
	where $\phi_{\beta} : \RR \rightarrow \RR$ is a fixed smooth function with the following properties : 
\begin{itemize}
\item[(\romannumeral1)] $\phi_{\beta}(x) \ge 0$, $\phi_{\beta}$ is even and convex.
\item[(\romannumeral2)] $\phi^{-1}_{\beta}(0) = I_{\beta - \frac{\pi}{2}} = [-(\beta - \frac{\pi}{2}), \beta - \frac{\pi}{2} ].$
\item[(\romannumeral3)] $\exists$ $a>0$ such that $\phi_{\beta}(x)>1$ if $x<-a$ or $x>a$.
\item[(\romannumeral4)] $\phi'_{\beta}(x) \neq 0$ if $\phi_{\beta}(x) = 1$.
\end{itemize}
\end{defn}

Let $\O$ be a domain in $\CC^3$ defined by the defining function
\[
	\rho(z,w,\gamma) :=  \left| z - e^{i \log|w|^2} \right|^2 - \left(1 - \phi_{\frac{3}{4}\pi}(\log|w|^2) - |\gamma|^2 \right),
\]
i.e., $\O := \{ (z,w,\gamma) \in \CC^3 : \rho(z,w,\gamma)  < 0 \}$, and
$\O_{\gamma} := \{ (z,w) \in \CC^2 : \rho_{\gamma}(z,w) := \rho(z,w,\gamma)  < 0 \}$ for each $\gamma \in \D$.
Here, the choice of $\beta = \frac{3}{4}\pi$ is not important in our example.

Now, for the following argument, we refer readers to \cite{diederich-fornaess2} and \cite{Krantz-Peloso}. 
Let 
\[
	\widetilde{\rho}_{\gamma}(z,w) := \rho_{\gamma}(z,w) e^{2 \arg w},
\]
which is a local defining function near any boundary point $p \in \partial \O_{\gamma}$.
Note that the holomorphic tangent plain $T^{1,0}_p(\partial \O_{\gamma})$ is spanned by a vector 
\[
	L_p := -\frac{\partial \rho_{\gamma}}{\partial w}(p) \left.\frac{\partial}{\partial z}\right|_p 
	+ \frac{\partial \rho_{\gamma}}{\partial z}(p) \left.\frac{\partial}{\partial w}\right|_p.
\]
Then
\begin{align*}
	\Levi_{\widetilde{\rho}_{\gamma}}(L_p, L_p) = \frac{e^{2\arg w}}{4|w|^2} 
	&\left[     
	\left| i\overline{z}(z - |w|^{2i}) + \phi'_{\frac{3}{4}\pi}(\log|w|^2) \right|^2  \right. \\
	& \left. +  \left|z - |w|^{2i} \right|^2 \left( \phi_{\frac{3}{4}\pi}(\log|w|^2) + \phi''_{\frac{3}{4}\pi}(\log|w|^2) \right) 
	\right].
\end{align*}
Observe that $\left| z - |w|^{2i} \right|$ never vanish on $\partial \O_{\gamma}$.
If $\log |w|^2 \notin [-(\beta - \frac{\pi}{2}), \beta - \frac{\pi}{2} ]$, then $\phi_{\frac{3}{4}\pi}(\log|w|^2) + \phi''_{\frac{3}{4}\pi}(\log|w|^2) > 0$ by the definition of $\phi_{\frac{3}{4}\pi}$. 
If $\log |w|^2 \in [-(\beta - \frac{\pi}{2}), \beta - \frac{\pi}{2} ]$, then $\left| i\overline{z}(z - |w|^{2i}) \right| > 0$ provided that $z \neq 0$.
Therefore, if $\gamma \neq 0$ then $z(p) \neq 0$, hence, $\Levi_{\widetilde{\rho}_{\gamma}}(L, L)(p) > 0$ for all $p \in \partial \O_{\gamma}$.
This means that $\O_{\gamma}$ is strongly pseudoconvex for all $\gamma \neq 0 \in \D$.
Since every strongly pseudoconvex domain admits a strictly plurisubharmonic defining function, $DF(\O_{\gamma}) = 1$ and $S(\O_{\gamma}) = 1$ for all $\gamma \neq 0 \in \D$.
Moreover, $\O_{0} = D_{\frac{3}{4}\pi}$ implies that $DF(\O_{0}) = \frac{2}{3}$ and $S(\O_{0})=2$ from the results in \cite{Liu1} and  \cite{Yum1}.
We conclude that $DF(\O_{\gamma})$ and $S(\O_{\gamma})$ are not upper semi-continuous and lower semi-continuous at $\gamma = 0$, respectively.


\end{document}